\documentclass[a4paper,12pt]{article}
\usepackage[utf8]{inputenc}
\usepackage{amsmath,amsthm,amssymb,amsfonts}
\usepackage{bbm,bm}
\usepackage{latexsym}
\usepackage{mathrsfs}
\usepackage{threeparttable}
\usepackage{tabularx}
\usepackage{booktabs}
\usepackage{graphicx,subfigure}
\usepackage{xcolor}
\usepackage{indentfirst}
\usepackage{geometry}
\usepackage{caption}
\usepackage{float}
\usepackage[numbers,sort&compress]{natbib}
\usepackage[colorlinks,
            linkcolor=blue,
            citecolor=blue
            ]{hyperref}
\usepackage{epstopdf}

\numberwithin{equation}{section}
\def \[{\begin{equation}}
\def \]{\end{equation}}

\newtheorem{thm}{Theorem}[section]

\newtheorem{defi}[thm]{Definition}
\newtheorem{claim}{Claim}
\newtheorem{lem}[thm]{Lemma}
\newtheorem{cor}[thm]{Corollary}
\newtheorem{ex}[thm]{Example}

\newtheorem{conj}[thm]{Conjecture}
\newtheorem{rem}[thm]{Remark}

\begin{document}

\synctex=1

\setlength{\baselineskip}{20pt}
\begin{center}{\Large \bf The resonance graphs of  coronoid systems and nanotubes}\footnote{This work is supported by NSFC\,(Grant No. 12271229).}

\vspace{4mm}

{Lingmei Liang, Heping Zhang\footnote{The corresponding author.}
\renewcommand\thefootnote{}\footnote{E-mail addresses: lianglm2023@lzu.edu.cn (L.Liang), zhanghp@lzu.edu.cn (H.Zhang).}}

\vspace{2mm}

\footnotesize{ School of Mathematics and Statistics, Lanzhou University, Lanzhou, Gansu 730000, P. R. China}

\end{center}

\noindent {\bf Abstract}:
The resonance graph of a hexagonal system  is connected, which shows that a perfect matching can be transformed into any other perfect matchings by a series of flips along hexagons. However, the resonance graph of a coronoid system (with holes) is not necessarily connected.
Saldanha et al. (Discrete
Comput. Geom. 14 (1995) 207-233) used homology and cohomology theory to obtain three versions of criteria for two tilings of a quadriculated region in the plane to be in the same connected component of the flip graph.  Inspiblack by the combinatorial version, in this paper we use a purely graph-theoretical approach to give a criterion in terms of simple invariant\textcolor{black}{---flow} across cuts between holes/exterior face for two perfect matchings of a coronoid system $G$ to be in the same connected component of its resonance graph. As a corollary we obtain a criterion for  the resonance graph of a coronoid system to be connected. We also discuss whether such  \textcolor{black}{criteria} are applicable to  nanotubes, and construct  a nanotube whose resonance graph is connected, which disproves a conjecture  proposed by Tratnik et al. (MATCH Commun. Math. Comput. Chem.  74 (2015) 175-186).

\vspace{2mm}
\noindent{\it Keywords}: Hexagonal system; Coronoid system; Nanotube; Perfect matching; Resonance graph
\vspace{2mm}


{\setcounter{section}{0}

\section{ Introduction}\setcounter{equation}{0}

Resonance graphs (or Z-transformation graph) originates from Herndon's resonance theory in 1973 \cite{Hern73} in chemistry and reflect the interactions between Kekul\'e structures of benzenoid hydrocarbons. This concept \textcolor{black}{ was independently introduced by chemists W. Gr$\rm{\ddot{u}}$ndler in \cite{Grun82} and S. El-Basil in \cite{ELBA93,ELBA932}, and by mathematicians Zhang et al. in \cite{ZGC88,ZGC88b}.} The degree sum of a resonance graph can be used to estimate the resonance energy of a benzenoid hydrocarbon \cite{ZLH93}.

A \emph{hexagonal system} (or \emph{benzenoid system})
is a finite 2-connected plane graph in which each interior face is a regular hexagon with  side length one. A \emph{coronoid system}  is a connected subgraph of a hexagonal system such that every edge belongs to a hexagon  and it contains a non-hexagon interior face (called hole). 
A hexagonal or coronoid system with a perfect matching (or Kekul\'e structures) can be viewed as the carbon-skeleton of a benzenoid or coronoid hydrocarbon \cite{CBCV91,CBCZ94}.

The resonance graph (or Z-transformation graph) of a hexagonal system as the graph on the set of perfect matchings: two vertices are adjacent
provided that their corresponding perfect matchings differ only in  one
hexagon (say a Z-transformation or flip along a hexagon). Afterwards, this concept was naturally extended to \textcolor{black}{ polyomino graphs} \cite{Zhang96}, fullerenes and nanotubes \cite{TZra16,TZra15,TZ16}, plane bipartite graphs \cite{ZZ2000, Four03}, plane graphs \cite{LWLZ26}
and graphs on surfaces \cite{TYra23} whenever flipped  hexagons are replaced with even interior face cycles. For an early survey on this topic, see \cite{Zhang06}.

For a hexagonal system,  it was proved that its resonance graph has the connectivity  equal to the minimum degree \cite{ZGC88b}.
For a polyomino graph, the same result holds \cite{Zhang96}  with two  exceptions.
For a general plane bipartite graph $G$ with a perfect matching, Zhang and Zhang \cite{ZZ2000} showed that the resonance graph $R(G)$ is connected if and only if $G$ is weakly elementary, i.e. every interior face of every elementary component of $G$ remains an interior face of the original $G$. In this case the oriented resonance graph implies a distributive lattice structure on the set of perfect matchings \cite{LZhang03,Zhang10} and $R(G)$ is a median graph \cite{ZLSh08, KZBL02}, which can be  embedded isometrically in a hypercube.
\textcolor{black}{\v{Z}igert Pleter\v{s}ek et al. \cite{ZBI12} showed that Lucas cubes are the nontrivial component of the resonance graphs of cyclic polyphenanthrenes.} Recently \v{Z}igert Pleter\v{s}ek \cite{Zige18} discoveblack that the resonance graph of a kinky benzenoid graph is a daisy cube \cite{KMla19}. Brezovnik et al. \cite{BCTr25} characterized all plane bipartite graphs whose resonance graphs are daisy cubes.

Liu et al. \cite{LWLZ26} presented a criterion for the resonance graph of a plane non-bipartite graph to be  connected. Under the condition that only squares are allowed to flip they further  showed  that the  resonance  graph of   cylindrical grid $P_{2m}\square C_{2n+1}$ is  connected, but showed \cite{LZZI25} that the resonance graph
of  toroidal grid $C_{2m}\square C_{2n+1}$ is not connected and consists of two isomorphic components, which imply that they  have the continuous forcing spectra (the forcing numbers of all perfect matchings form an integer interval).

For a coronoid system $G$ with a perfect matching $M$, an $M$-alternating hexagon can be flipped to obtain another perfect matching. Such a hexagon is an aromatic sextet (also resonant), which is more important than longer resonant cycle in chemical stability. So the holes and outer face are not allowed to flip. So we give the following definition.
\begin{defi}
{\rm  The resonance graph (also called Z-transformation graph)
 of a hexagonal or coronoid system $G$, denoted by $R_6(G)$, is a simple graph in which the vertices are the perfect matchings of $G$ and
 two perfect matchings $M_1$ and $M_2$ are joined by an edge provided
 their symmetric difference consists exactly of six edges of one hexagon of $G$ (say $M_1$ and $M_2$ are transformed into each other by a flip along the hexagon).}
\end{defi}

 The following  example shows that the resonance graph of a coronoid system $G$ is not necessarily connected. It is natural to ask whether any given two perfect matchings of a coronoid system $G$ lie in the same component of it resonance graph $R_6(G)$.
\begin{figure}
\centering
\includegraphics[scale=0.8]{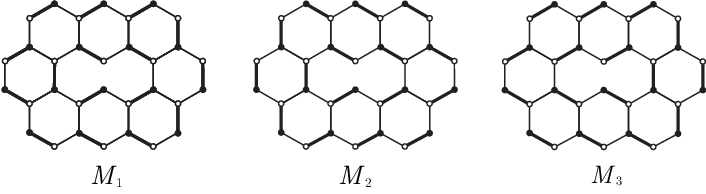}
\caption{\label{Figure_1} Three perfect matchings (bold edges) of a coronoid system $G$.}
\end{figure}

\begin{ex}
{\rm  Let $G$ be coronoid system with one hole shown in Fig.  \ref{Figure_1}. A routine computation shows that $G$ has exactly 40 perfect matchings and   the resonance graph $R_6(G)$ consists of one large component containing $M_{1}$ and two singleton components $M_{2}$ and $M_{3}$.}
\end{ex}

For a quadriculated region $P$ in the plane consisting of squares, a \emph{domino} is the union of two adjacent
squares and a (domino) \emph{tiling} is a collection of dominoes with disjoint interior whose
union is the region. The \emph{flip graph} of $P$ is a graph on the set of all tilings of $P$ such  that two tilings are adjacent if we change one to another by a flip (a $90^{\circ}$
rotation of a pair of side-by-side dominoes), which corresponds to the resonance graph of the inner dual of $P$ (only unit squares are allowed to flip).
Saldanha et al. \cite{STCR95} used homology and cohomology theory to obtain combinatorial, homological and height section  versions of criteria for two tilings of a quadriculated region to be in the same connected component of the flip graph. They showed such three versions are equivalent. For combinatorial version, Saldanha et al. used   a quite  simple invariant of a domino tiling - {\sl flow} across a cut between holes derived by 1-dimensional cohomology class corresponding to the cut.

Motivated by the above combinatorial version, in this paper we define  flows across $n$ cuts between holes (or exterior face) of a coronoid system $G$ with $n$ holes which jointly connect all holes and exterior face and do not disconnect $G$, and  show that two perfect matchings of a coronoid system $G$ to belong to the same
connected component of the resonance graph of $G$ if and only if they have the same  flows  across  all such cuts (see Theorem \ref {main results}). In the next section, we present some preliminary concepts and results.  In Section 3, we use a  graph-theoretical method to  prove the main result. In Section 4, we give two  applications.  Firstly we show that  the resonance graph of a coronoid system is connected if and only if each nice cycle bounds a hexagonal system. Secondly, we point out that this criterion holds for nanotubes, and the flow-criterion is still effect for elementary nanotubes.   Finally, we give a counterexample to the conjecture proposed by Tratnik et al. \cite{TZra15} that states that the resonance graph of a nanotube is not connected.

\section{ Preliminaries}

In this section, some definitions and useful results are given first.  For a graph $G$, let $V(G)$ and $E(G)$ denote its vertex set and edge set, respectively. For a bipartite graph $G$, its vertices are always coloblack white and black such that any pair of adjacent vertices receive different colors.  \textcolor{black}{Coronoid systems and nanotubes are bipartite, respectively, see \cite{CBCZ94} and \cite{SHZ96}.}

\begin{color}{black}
By the definition we know that a coronoid system $G$ is a 2-connected plane graph,  the boundaries of all nonhexagonal faces are disjoint cycles, and the the boundaries of all holes lie in the interior of the boundary of $G$.  The  following lemma shows the converse also holds.
\begin{lem}\label{equivalent}
 Let $C_0, C_1, \dots, C_n$ ($n \geq 1$) be disjoint cycles on a hexagonal lattice such that the interior of each $C_i$, where $1\leq i\leq n$, contains at  least two hexagons and is entirely contained in the interior of $C_0$. Then the graph  $G$  obtained from $I[C_0]$ by deleting the interiors of $C_1,C_2,\dots,C_n$ is a coronoid system.
\end{lem}

\begin{proof} $G$ is a plane graph and each face is bounded by a hexagon or cycle $C_i$, $0\leq i\leq n$. Then $G$ is 2-connected. For any edge $e$ of $G$, if $e$ lies in  $C_0$, let $h$ denote a hexagon of $I[C_0]$ such that $e$ is an edge of $h$. Since $C_0$ is disjoint with each $C_i$, $1\leq i\leq n$,  $h$ does not lie in the interior of each $C_i$,  so $h$ is a hexagon of $G$; if $e$ lies in  $C_i$, $1\leq i\leq n$, let $h$ denote a hexagon of $O[C_i]$ such that $e$ is an edge of $h$. Since $C_i$ is disjoint with the other $C_j$, $1\leq j\not=i$,  $h$ does not lie in the interior of each such $C_j$,  so $h$ is a hexagon of $G$. Otherwise, $e$ does not lie in any $C_i$. In this case both hexagons with edge $e$ in $I[C_0]$ are hexagons of $G$. Hence $G$ is a coronoid system.
\end{proof}

\end{color}

\textcolor{black}{A \emph{matching} $M$ of a graph $G$ is an edge subset such that no two edges of $M$
have a common end-vertex.
A \emph{perfect matching} of $G$ is a matching that covers
all the vertices of $G$.
A graph $G$ is \emph{matchable} if it has a perfect matching \cite{LMuu24}.}  An edge of  a matchable bipartite  graph $G$ is said to be \emph{forbidden} (or fixed single) if it does not lie in any perfect matching of $G$.  The \emph{elementary components}
of $G$ mean components of the subgraph obtained from $G$ by the removal of all forbidden edges. A connected bipartite graph $G$ is  called \emph{elementary} if it admits a perfect matching and has no  forbidden edges.

For a perfect matching $M$ of a plane bipartite graph $G$, a cycle $C$ (or path $P$) is called an \emph{$M$-alternating cycle} (or path) if the edges of $C$ (or $P$) appear alternately in $M$ and off $M$.
An $M$-alternating cycle $C$ of  $G$ is said to be \emph{proper (improper)} if every edge of $C$ belonging to $M$ goes from the white (black) end-vertex to the black (white)
end-vertex by the clockwise orientation of $C$. \textcolor{black}{A face of $G$ is said to be \emph{proper (improper) $M$-alternating face} if its boundary is a proper (improper) $M$-alternating cycle.} 
 \textcolor{black}{A cycle of $G$ is \emph{nice} if $G-V(C)$ has a perfect matching. Equivalently, a cycle  of $G$ is nice (or {\it resonant}) if $G$ has a perfect matching $M$ such that $C$ is  $M$-alternating.  A face of $G$ is said to be \emph{resonant} if its boundary is a resonant cycle.}

 For sets $A$ and $B$, the \emph{symmetric difference} of $A$ and $B$ is $A\triangle B=(A\cup B)\setminus(A\cap B)$.
For an $M$-alternating cycle $C$ of $G$, $E(C)\triangle M$ is also a perfect matching of $G$. For  distinct  perfect matchings $M$ and $M'$ of  $G$,  $M\triangle M'$ forms disjoint
 $(M,M')$-alternating cycles of $G$.

We also present some basic properties of hexagonal systems,
plane elementary bipartite graphs and that make use to our proofs as follows:
\begin{lem}[\cite{ZGC88}]\label{hexagonal system}
Let $H$ be a matchable hexagonal system. Then the resonance graph of $H$ is a
connected bipartite graph.
\end{lem}

\begin{lem}[\cite{ZZ2000}]\label{nice cycle}
Let $G$ be a plane bipartite graph
with more than two vertices. Then
each face of $G$ is resonant if
and only if $G$ is elementary.
\end{lem}

In \cite{ZZ1992}, Zhang and Zheng  gave the following criterion for every hexagon of a coronoid system being resonant. For other mathematical properties on perfect matchings of coronoid systems, see monograph \cite{CBCZ94}.
\begin{lem}[\cite{ZZ1992}]\label{elementary}
Every hexagon of a hexagonal or coronoid system is resonant if and only if each non-hexagonal face is resonant.
\end{lem}
For a cycle $C$ of a plane graph $G$, let $I[C]$ (resp. $O[C]$) denote the subgraph of $G$ consisting of $C$ and its interior (resp. exterior).
\begin{color}{black}
\begin{lem}[\cite{ZZ2000,ZZYh04}]\label{I[C] elementary}
For a nice cycle $C$ of a plane bipartite elementary graph, both $I[C]$ and $O[C]$ are elementary.
\end{lem}
\begin{lem}[\cite{ZZ1999}]\label{$M$-alternating face}
Let $G$ be a plane elementary bipartite graph and $M$ be a perfect matching of $G$. Then the interior of each proper (improper) $M$-alternating cycle of $G$, there must exist a proper (improper) $M$-alternating face.
\end{lem}
\end{color}

\textcolor{black}{Let $G$ be a plane connected bipartite graph, and $F(G)$ be the set of all faces of $G$.
For any $F \subseteq F(G)$, the restricted resonance graph $R_{F}(G)$ of $G$ with respect to $F$
is a graph whose vertices are the perfect matchings of $G$, and two perfect matchings $M_1, M_2$ are joined by an edge  provided their symmetric difference $M_1 \triangle M_2$ forms exactly one cycle that is the boundary of a face in  $F$. When $F$ is the set of all interior faces of $G$,  $R_{F}(G)$ is coincides with the resonance graph of $G$, denoted by $R(G)$.}

The \emph{\textcolor{black}{Cartesian} product} of simple graphs $G$ and $H$ is the graph $G \square H$ whose
vertex set is $V(G) \times V(H)$ and vertices $(u_1, v_1)$ and $(u_2, v_2)$ are adjacent if and only if  either $u_1u_2 \in E(G)$ and $v_1 = v_2$ or $u_1 = u_2$ and $v_1v_2 \in E(H)$.

\begin{lem}[\cite{ZZYh04}]\label{component}
 Let $G$ be a matchable plane bipartite graph and $F \subseteq F(G)$. Let $G_1,\ldots,G_k$ denote the elementary components of $G$. Then
$$R_{F}(G)\cong R_{F_{1}}(G_1)\square \cdots \square R_{F_{k}}(G_k),$$ where $F_{i}=F(G_i) \cap F,  i = 1, \ldots, k$.
\end{lem}

\subsection{Flow across a cut}

For a coronoid system $G$ with $n$ holes, fix a white-black vertex coloring of $G$. For convenience,  let $h_{1},h_{2},\ldots,h_{n}$ be $n$ holes and $h_0$ be the exterior face of $G$. So $h_0,h_1,\ldots,h_n$ are all the non-hexagonal faces of $G$. The \emph{dual graph} $G^{*}$ of $G$ is a plane graph that has  one vertex $h^{*}$ at the center of each hexagon $h$   and one vertex $h_i^*$ in the interior of each non-hexagon face $h_i$,   two vertices $f_{1}^{*}$ and $f_{2}^{*}$ are joined by the edge $e^{*}$ in $G^{*}$  if their corresponding faces $f_{1}$ and $f_{2}$ \textcolor{black}{share a common edge $e$} ($e^{*}$ only crosses edge $e$ at the center, $e^*$ is straight in the interiors of hexagons, \textcolor{black}{but not necessarily straight in the interiors of non-hexagonal faces}).

We first define a cut segment of a coronoid system $G$ between non-hexagon faces.
\begin{defi}\label{2.10}
{\rm  \textcolor{black}{A directed path $L$ is called  a \emph{cut segment} of a coronoid system $G$ if its underlying graph is a path of  the dual graph $G^*$   such that  two end vertices correspond to two distinct non-hexagonal faces of $G$ and  the internal vertices correspond to  hexagons of $G$.
The set of edges of $G$ crossed by $L$ is called a \emph{cut}, denoted by $L^*$.}}
\end{defi}
\begin{figure}
\centering
\includegraphics[scale=0.7]{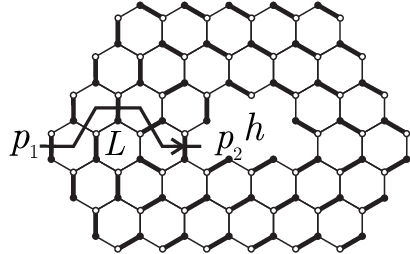}
\caption{\label{Figure_2} A coronoid system  with a cut segment $L$ and  a perfect matching $M$.}
\end{figure}
\begin{defi}\label{2.11}
{\rm    The \emph{flow} of a given perfect matching $M$ of $G$ across a cut segment $L$, denoted by flow$_{L}(G, M)$, is the number of  edges in $M$ crossed by $L$, where each such edge in $M$ is
counted positively (resp. negatively) if its white end-vertex is to the left side (resp. right side) of the cut segment $L$.}
\end{defi}

For convenience,  let $(L^{*})^{+}=\{e\in L^*|\mbox{the white end-vertex of $e$ is to the left side of  $L$}\}$ and $(L^{*})^{-}=\{e\in L^*|\mbox{the white end-vertex of $e$ is to the right side of  $L$}\}$.
Then
\begin{equation}\label{dif}
    {\rm flow}_{L}(G, M)=|(L^{*})^{+}\cap M|-|(L^{*})^{-}\cap M|.
\end{equation}

 For example, for coronoid system $G$ in Fig. \ref{Figure_2} with a cut segment $L$ and perfect matching $M$ we have  flow$_{L}(G, M)=3-1=2$.

Let $G_1$ be a subgraph of $G$ such that $M\cap E(G_1)$ is a perfect matching of $G_1$. We use ${\rm flow}_{L}(G_{1}, M)$
to denote \textcolor{black}{the flow of the restricted  $M|_{G_1}:=M\cap E(G_1)$ across a cut segment $L$. That is,
\begin{equation}\label{rest}
    {\rm flow}_{L}(G_1, M)=|(L^{*})^{+}\cap M|_{G_1}|-|(L^{*})^{-}\cap M|_{G_1}|.
\end{equation}}
Hence we have
\begin{equation}
      {\rm flow}_{L}(G, M)={\rm flow}_{L}(G_1, M)+{\rm flow}_{L}(G-V(G_1), M).
\end{equation}

\begin{lem}\label{tree}
\begin{color}{black}Let $G$ be a  coronoid system with $n$ holes. Then $G^*$ admits
a  tree $T$ that  contains all vertices  corresponding to the non-hexagonal faces of $G$ and each leaf corresponds to a non-hexagonal face. 
\end{color}
\end{lem}
\begin{proof}
\begin{color}{black}Since the dual graph $G^{*}$ is connected, $G^{*}$ has a spanning tree $T_0$. Starting from $T_0$,  delete repeatedly a leaf  corresponding to a hexagon of $G$.  The resulting graph is still a tree, say $T$.
Then every leaf of $T$ corresponds to a
non-hexagonal face of $G$. Since no vertex corresponding to non-hexagonal face has been deleted, $T$ contains all vertices  corresponding to non-hexagonal faces of $G$. \end{color}
\end{proof}
\begin{color}{black}
We now give an orientation to the tree $T$ such that the path from  $h_0^{*}$ to each other vertex  of $T$ is a directed path. The resulting digraph is a rooted tree, also denoted by $T$, where $h_0^*$ is the root, a unique vertex of in-degree 0, and the all leaves are of in-degree 1 and out-degree 0.  

From the rooted tree $T$ we choose $n$ cut segments of $G$. For each $1\leq i\leq n$, since $T$ has a unique directed path from the root to $h_i^*$, take the directed subpath, say $L_i$, from $h_j^*$ to $h_i^*$ whose internal vertices correspond to hexagons of $G$ and $h_j^*$ corresponds to a non-hexagonal face $h_j$ of $G$, $j\not=i$.
Thus, $L_i$ is a cut segment in $G$. We say that   $L_1,\ldots,L_n$ are {\it chosen $n$ cut segments} in  $G$. Since $T$ is a tree of $G^*$ connecting all $h_i^*$, by duality $G-\cup_{i=1}^{n}L_i^*$ is a connected subgraph in which each interior face is a hexagon.
\end{color}

\textcolor{black}{It is possible that the two cut segments of the chosen $n$ cut segments have a common part at start;  For example, see  $L_{1}$ and $L_{2}$ in Fig. \ref{Figure_3}(b).}
\begin{figure}
\centering
\includegraphics[scale=0.8]{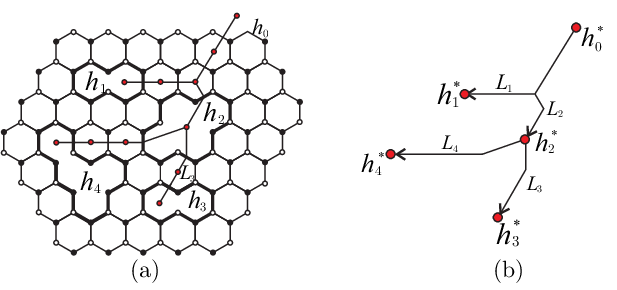}
\caption{\label{Figure_3} (a) A coronoid system $G$ and a tree in  $G^{*}$, (b) four chosen cut segments $L_{1}, L_{2},L_{3}$  and $L_{4}$ in a rooted tree $T$. }
\end{figure}

\subsection{Main result}
Based on the above concepts and results, we \textcolor{black}{now provide a criterion for determining when} two perfect matchings of a coronoid system $G$ belong to the same connected component of the resonance graph $R_6(G)$, \textcolor{black}{using the flow of perfect matchings across a cut segment.}

\begin{thm}\label{main results}
Let $G$ be a matchable coronoid system with $n$ holes and  chosen $n$ cut segments $L_i, 1\leq i\leq n$. Two perfect matchings of $G$ lie in the same connected component of the resonance graph $R_6(G)$ if and only if their flows across each of the $n$ chosen cut segments are equal.
\end{thm}

\section{Proof of Theorem \ref{main results}}
 Throughout this section, let $G$ be a coronoid system with $n$ holes and  $n$ chosen cut segments  $L_{1},L_{2},\ldots,L_{n-1},L_{n}$. \textcolor{black}{For convenience, we call a directed path of the rooted tree $T$ from some $h^*_i$ to $h_j^*$ a {\it cut line} of $G$, which may be divided into several cut segments.}

First, we show some lemmas, which will play a crucial role in the proof of Theorem \ref{main results}. The following lemma explains the variation between the flows of two perfect matchings $M_{1}$, $M_{2}$ \textcolor{black}{across a cut line of $G$ restricted on an $(M_{1},M_{2})$-alternating cycle  (this definition is an extension of  Definition \ref{2.11}). }

\begin{lem}\label{one II cycle}
Let $M_{1}$, $M_{2}$ be  perfect matchings of $G$,  $C$  an $(M_{1},M_{2})$-alternating cycle, \textcolor{black}{and $L$ a cut line. If both endpoints of $L$ are either in the interior of $C$ or in the exterior of $C$}, then ${\rm flow}_{L}(C,M_{1})={\rm flow}_{L}(C,M_{2})$. Otherwise, $|{\rm flow}_{L}(C,M_{1})-{\rm flow}_{L}(C,M_{2})|=1$.
\end{lem}
\begin{proof}
If both endpoints of $L$ are either in the interior of $C$ or in the exterior of $C$ (see Fig. \ref{Figure_4}(a)),
then $|L^{*}\cap E(C)|$ is even. If $|L^{*} \cap E(C)|=0$, it is trivial.
Otherwise, let $L^{*}\cap E(C)=\{e_{1},e_{2},\ldots,e_{2m}\}$, which
are sequentially labeled  along the clockwise orientation of $C$. For any two successive edges $e_{i},e_{i+1}$ (the subscripts modulo $2m$), we discuss their effect on the flows of $M_{1}$ and $M_{2}$ across $L$.

 {\bf{ Case 1.}} The white end-vertices of  $e_{i},e_{i+1}$ are on the same side of $L$.

 If the white end-vertices of $e_{i}$ and $e_{i+1}$ are on the left of $L$ (see $e_{2m-1}$ and $e_{2m}$ in Fig. \ref{Figure_4}(a)), then $e_{i},e_{i+1}\in (L^{*})^{+}$ and $(L^{*})^{-}\cap \{e_{i},e_{i+1}\}=\emptyset$. Since an $(M_{1},M_{2})$-alternating path in $C$ with end-edges  $e_{i}$ and $e_{i+1}$ is of even length,  one of $e_{i}$ and $e_{i+1}$ is in $M_{1}$, and the other one is in $M_{2}$. Then $|(L^{*})^{+}\cap E(C)\cap M_{1}\cap \{e_{i},e_{i+1}\}|=|(L^{*})^{+}\cap E(C)\cap M_{2}\cap \{e_{i},e_{i+1}\}|=1$. If
the white end-vertices of $e_{i}$ and $e_{i+1}$ are on the right of $L$,
similarly we have  $e_{i},e_{i+1}\in (L^{*})^{-}$ and $(L^{*})^{+}\cap \{e_{i},e_{i+1}\}=\emptyset$, and
$|(L^{*})^{-}\cap E(C)\cap M_{1}\cap \{e_{i},e_{i+1}\}|=|(L^{*})^{-}\cap E(C)\cap M_{2}\cap \{e_{i},e_{i+1}\}|=1$.
\begin{figure}
\centering
\includegraphics[scale=0.8]{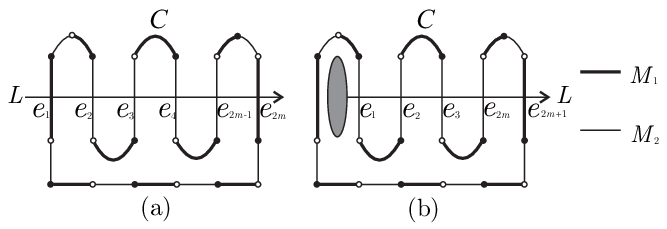}
\caption{\label{Figure_4}Illustration for the proof  of Lemma \ref{one II cycle}.}
\end{figure}

{\bf{ Case 2.}} The white end-vertices of  $e_{i},e_{i+1}$ are not on the same side of $L$.

Without loss of generality, let $e_{i}\in (L^{*})^{+}$ and $e_{i+1}\in (L^{*})^{-}$ (see  Fig. \ref{Figure_4}(a) when $i=3$). Since  an $(M_{1},M_{2})$-alternating path in $C$ with end-edges  $e_{i}$ and $e_{i+1}$ is of odd length,   either both $e_i$ and $e_{i+1}$ lie in $M_1$, or both lie in $M_2$.  If $e_{i},e_{i+1}\in M_{1}$, then $|(L^{*})^{+}\cap E(C)\cap M_{1}\cap \{e_{i},e_{i+1}\}|-|(L^{*})^{-}\cap E(C)\cap M_{1}\cap \{e_{i},e_{i+1}\}|=1-1=0$, and $|L^{*})^{+}\cap E(C)\cap M_{2}\cap \{e_{i},e_{i+1}\}|-|(L^{*})^{-}\cap E(C)\cap M_{2}\cap \{e_{i},e_{i+1}\}|=0-0=0.$
 If $e_{i},e_{i+1}\in M_{2}$, then the above equalities also hold whenever $M_1$ and $M_2$ interchange.

In short, for each $i$ we have $|(L^{*})^{+}\cap E(C)\cap M_{1}\cap \{e_{i},e_{i+1}\}|-|(L^{*})^{-}\cap E(C)\cap M_{1}\cap \{e_{i},e_{i+1}\}|=|(L^{*})^{+}\cap E(C)\cap M_{2}\cap \{e_{i},e_{i+1}\}|-|(L^{*})^{-}\cap E(C)\cap M_{2}\cap \{e_{i},e_{i+1}\}|$.
So $${\rm flow}_{L}(C,M_{j})=\sum\limits_{i=1}^m(|(L^{*})^{+}\cap E(C)\cap M_{j}\cap \{e_{2i-1},e_{2i}\}|-|(L^{*})^{-}\cap E(C)\cap M_{j}\cap \{e_{2i-1},e_{2i}\}|),$$ for $j=1,2$, which implies ${\rm flow}_{L}(C,M_{1})={\rm flow}_{L}(C,M_{2})$.

If one of the two endpoints of $L$ is in the exterior of $C$ and the other one is in the interior of $C$,  then $|L^{*}\cap E(C)|$ is odd (see Fig.  \ref{Figure_4}(b)). Similarly, let $L^{*}\cap E(C)=\{e_{1},e_{2},\ldots,e_{2m}, e_{2m+1}\}$. For \textcolor{black}{the first $m$ pairs of edges,} the above arguments are still correct, and the last edge $e_{2m+1}$ belongs to either $M_1$ or $M_2$. So   ${\rm flow}_{L}(C,M_{1})={\rm flow}_{L}(C,M_{2})\pm   1$.
\end{proof}
A cycle of a coronoid system $G$ is called an I-cycle if the interior of $C$  contains no hole of $G$, and a II-cycle otherwise.
It is evident that  the boundaries of holes and infinite face in $G$ are II-cycles.
\begin{lem}\label{two cycle}
Let $M_{1}$,
$M_{2}$ be any two perfect matchings of $G$ and \textcolor{black}{$L$ a cut line of $G$.}
For two disjoint  $(M_{1},M_{2})$-alternating II-cycles $C_{1}$ and $C_{2}$  such that $I[C_{2}]\subset I[C_{1}]$, if ${\rm flow}_{L}(C_{1}\cup C_{2},M_{1})={\rm flow}_{L}(C_{1}\cup C_{2},M_{2})$ and the endpoints of $L$ lie in the interior of $C_{2}$ and in the exterior of  $C_{1}$ separately, then one of $C_{1}$ and $C_{2}$ is a proper $M_{1}$-alternating cycle and
the other one is improper.
\end{lem}

\begin{proof}
By Lemma \ref{one II cycle}, $|{\rm flow}_{L}(C_{j},M_{1})-{\rm flow}_{L}(C_{j},M_{2})|=1$, $j=1,2$.
By the proof of Lemma \ref{one II cycle}, we know that ${\rm flow}_{L}(C_{j},M_{1})-{\rm flow}_{L}(C_{j},M_{2})$ is determined by the last edge $e_j$ of $C_{j}$  crossed by $L$  along the clockwise orientation of $C_{j}$.  Without loss of generality suppose that the start point of $L$ lies in the exterior of $C_1$. Suppose to the contrary that $C_{1}$ and $C_{2}$ are proper $M_{1}$-alternating cycles (the proof for improper case is similar).
If the white end-vertices of $e_{1},e_{2}$ are on the same side of $L$, then either $e_{1}$ and $e_{2}$  both are in $M_{1}$, or  both in $M_{2}$ (see Fig. \ref{Figure_5}(a) and (b)). We can see that ${\rm flow}_{L}(C_{1},M_{1})-{\rm flow}_{L}(C_{1},M_{2})=-1$ and ${\rm flow}_{L}(C_{2},M_{1})-{\rm flow}_{L}(C_{2},M_{2})=-1$. If the white end-vertices of two edges $e_{1},e_{2}$ are not on the same side of $L$, then one of them is in $M_{1}$, and the other one is in $M_{2}$ (see Fig. \ref{Figure_5}(c) and (d)). We also have ${\rm flow}_{L}(C_{1},M_{1})-{\rm flow}_{L}(C_{1},M_{2})={\rm flow}_{L}(C_{2},M_{1})-{\rm flow}_{L}(C_{2},M_{2})=-1$.
Hence ${\rm flow}_{L}(C_{1}\cup C_{2},M_{1})-{\rm flow}_{L}(C_{1}\cup C_{2},M_{2})=-2$, a contradiction.
\begin{figure}
\centering
\includegraphics[scale=0.8]{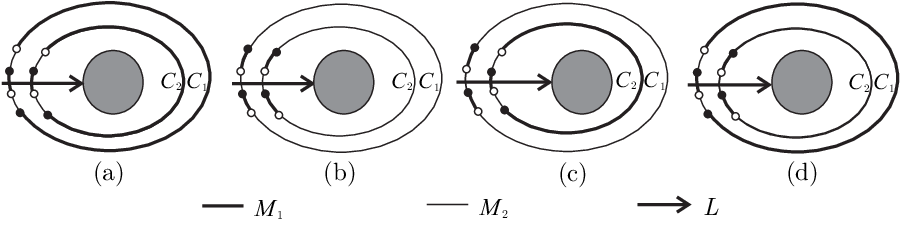}
\caption{\label{Figure_5} Illustration for Lemma \ref{two cycle} with its proof.}
\end{figure}
\end{proof}


Let $\mathcal{C}$ be the union of all $(M_{1},M_{2})$-alternating cycles of
$M_{1}\triangle M_{2}$ and $\mathcal{C'}$ be the union of all II-cycles of $M_{1}\triangle M_{2}$. Let $E(\mathcal{C})$ (resp. $E(\mathcal{C}')$)  be the  set of all edges in $\mathcal{C}$ (resp. $\mathcal{C}'$). We say that two perfect matchings $M_1$ and $M_2$ are {\sl balanced around a hole} $h$ of $G$ if among  $\mathcal{C}$ there are the same number of proper and improper $M_1$-alternating cycles which contain $h$ in their interior.
\begin{lem}\label{EC}
Let $M_{1}$, $M_{2}$ be  two perfect matchings of $G$ and \textcolor{black}{$L$ be a cut line of $G$}.
Then ${\rm flow}_{L}(G,M_{1})={\rm flow}_{L}(G,M_{2})$ if and only if ${\rm flow}_{L}(\mathcal{C'}, M_{1})={\rm flow}_{L}(\mathcal{C'},M_{2})$. Moreover, if ${\rm flow}_{L}(G,M_{1})={\rm flow}_{L}(G,M_{2})$, then $|L^{*}\cap E(\mathcal{C'})|$ is even.
\end{lem}
\begin{proof}
Since $M_{1}\triangle M_{2}=\mathcal{C}$, ${\rm flow}_{L}(G-V(\mathcal{C}),M_{1})={\rm flow}_{L}(G-V(\mathcal{C}),M_{2})$.
By ${\rm flow}_{L}(G,M_{i})=
{\rm flow}_{L}(G-V(\mathcal{C}), M_{i})+{\rm flow}_{L}(\mathcal{C}, M_{i})$, $i=1,2$, we have ${\rm flow}_{L}(G,M_{1})={\rm flow}_{L}(G,M_{2})$ if and only if
${\rm flow}_{L}(\mathcal{C},M_{1})={\rm flow}_{L}(\mathcal{C},M_{2})$.
For any I-cycle $C$ of $\mathcal{C}$,
since both two endpoints of $L$ are in the exterior of $C$, by Lemma \ref{one II cycle}, we have ${\rm flow}_{L}(C,M_{1})={\rm flow}_{L}(C,M_{2})$.
Thus, ${\rm flow}_{L}(G,M_{1})={\rm flow}_{L}(G,M_{2})$ if and only if ${\rm flow}_{L}(\mathcal{C'},M_{1})={\rm flow}_{L}(\mathcal{C'},M_{2})$.

If ${\rm flow}_{L}(G,M_{1})={\rm flow}_{L}(G,M_{2})$, then ${\rm flow}_{L}(\mathcal{C'},M_{1})-{\rm flow}_{L}(\mathcal{C'},M_{2})=0$. Let $\mathcal{C}'_{1}$ be the set of cycles $C$ in  $\mathcal{C'}$ such  that   the two endpoints of $L$ belong to   the exterior and interior of $C$ separately.  For any cycle $C$ of $\mathcal{C'}$ not in $\mathcal{C}'_{1}$,
by Lemma \ref {one II cycle},  $|L^{*}\cap E(C)|$ is even and ${\rm flow}_{L}(C,M_{1})={\rm flow}_{L}(C,M_{2})$.
Then  \begin{equation}\label{3.1}
    \sum\limits_{C\in \mathcal C_1'}({\rm flow}_{L}(C,M_{1})-{\rm flow}_{L}(C,M_{2}))=0.
\end{equation}
For $C\in \mathcal C_1'$, by Lemma \ref {one II cycle}, $|L^{*}\cap E(C)|$ is odd and $|{\rm flow}_{L}(C,M_{1})-{\rm flow}_{L}(C,M_{2})|=1$. The latter and Eq. (\ref{3.1}) imply that  $|\mathcal C'_1|$ is even. Moreover, 
$|L^{*}\cap E(\mathcal{C'})|=|L^{*}\cap E(\mathcal{C}'_{1})|+|L^{*} \cap E(\mathcal{C'}-\mathcal{C}'_{1})|$ is even.
\end{proof}


\begin{cor}\label{Necessity}
Let $G$ be a matchable coronoid system. If  perfect matchings $M_1$ and $M_2$ of $G$ lie in the same connected component of $R_{6}(G)$, then their flows across a cut segment $L$ are equal.
\end{cor}

\begin{proof}
$R_6(G)$ has a path $P' = M_{1}'
(= M_1)M_{2}'\cdots M_{t-1}'M_{t}'(= M_{2})$
between $M_{1}$ and $M_{2}$. For any consecutive perfect matchings $M_i'$ and $M_{i+1}'$,   $i=1,2,\ldots,t-1$, $M_{i}{'}\triangle M_{i+1}'$ forms a hexagon $S$ of $G$, which is an I-cycle. \textcolor{black}{So by Lemma \ref{one II cycle}, we have ${\rm flow}_{L}(S, M_{i}')= {\rm flow}_{L}(S,M_{i+1}')$. By Lemma \ref{EC} we have ${\rm flow}_{L}(G, M_{i}')= {\rm flow}_{L}(G,M_{i+1}')$.} Hence
 ${\rm flow}_{L}(G,M_{1})= {\rm flow}_{L}(G,M_{2})$.\end{proof}

Corollary \ref{Necessity} shows the necessity of Theorem \ref{main results}.  To show its sufficiency, we first describe some basic results. \textcolor{black}{The following result can be obtained from the  proof of Lemma 3.3 in \cite{LZhang03}. Here we give a new and simpler proof.}
\begin{lem}\label{path}
\begin{color}{black}
Let $G$ be a plane elementary bipartite graph,  $C_0$ be the boundary of the exterior face of $G$ and $C_1,\ldots,C_n$  be the boundaries of interior faces of $G$.  Let  $M_{1}$ be a  perfect matching of $G$ such that $C_0$ is  a proper $M_{1}$-alternating cycle and $C_{1},\ldots,C_{n}$ are improper  $M_{1}$-alternating cycles. Then
 $R(G)$ has a path between $M_{1}$ and $M_{2}:=M_1\triangle \bigcup_{i=0}^n E(C_{i})$  by flipping only those faces of $G$ in the region $R$, where $R$ is obtained from the interior of  $C_{0}$ minus  $C_{i}$ ($1\leq i\leq n$) with their interiors.
\end{color}
\end{lem}
\begin{proof}
\begin{color}{black}We can see that the $C_i$'s are mutually disjoint by the condition in the lemma, which  shows that  $R$ is indeed a region. Let $s$ be the number of faces of $G$  lying  in   $R$.  We prove, by induction on $s$, that  $R(G)$ has a path between $M_{1}$ and $M_{2}$ by flipping only those  faces of $G$  in the region $R$. If $s=1$, then $G=C_{0}$ and it is trivial.

Suppose $s\geq 2$. Since $C_{0}$ is a proper $M_{1}$-alternating cycle of $G$,
by Lemma \ref{$M$-alternating face}, $G$ has a proper $M_{1}$-alternating face $f$. Hence $f$ lies in the interior of $C_0$ and on the exterior of each $C_i$, $1\leq i\leq n$. Let  $C$ be the boundary of $f$ and $R'$ the point set obtained from $R$ by deleting  $f$ and its boundary. Note that $C$ may share an edge of some $C_i$'s. Then $R'$ is either a region or the union of several separated regions and  $(\cup_{i=0}^n E(C_{i}))\triangle E(C)$ form the entire boundary of $R'$.  
Let $M_{1}':=M_{1}\triangle E(C)$. Then $M_{1}'$ is a perfect matching of $G$ and
$M_{1}'\triangle M_{2}=M_{1}\triangle E(C) \triangle M_{2}=(\cup_{i=0}^n E(C_{i}))\triangle E(C)$. For each component (region) of $R'$, the outer boundary and the inner boundaries are  proper and improper $M'_1$-alternating cycles (improper and proper $M_2$-alternating cycles),  respectively.

By using repeatedly the induction hypothesis, $R(G)$ has a path $P$ between $M_{1}'$ and $M_{2}$ only by flipping faces in  $R'$. Since $M_{1}':=M_{1}\triangle E(C)$, $R(G)$ has an edge  between $M_{1}$ and $M_{1}'$ by flipping face $f$. Thus $R(G)$ has a path  between $M_{1}$ and $M_{2}$ by flipping faces in the region $R$.
\end{color}
\end{proof}


We are in position to show the sufficiency of Theorem \ref{main results} (Lemma \ref{Sufficiency} ) by using the above lemmas.
\begin{lem}\label{Sufficiency}
Let $G$ be a matchable coronoid system with $n$ holes. 
For  perfect matchings  $M_{1}$ and $M_{2}$ of $G$, if ${\rm flow}_{L_{i}}(G,M_{1})= {\rm flow}_{L_{i}}(G,M_{2})$ for each cut segment $L_{i}$ ($1\leq i\leq n$) of  chosen $n$ cut segments,
then $R_{6}(G)$ has a path between $M_{1}$ and $M_{2}$.
\end{lem}
\begin{proof}
Let $\mathcal{C}=M_{1}\triangle M_{2}=\cup_{i=1}^k C_{i}$, $k\geq0$,
where the $C_{i}$'s are  disjoint $(M_1,M_2)$-alternating cycles.
We now
show that $R_{6}(G)$ has a path between $M_{1}$ and $M_{2}$ by induction on the number $k$. If $k=0$, it is trivial. Next, let $k\geq 1$. If $\mathcal{C}$ contains an I-cycle $C$, let $M_{3}=M_{1}\triangle E(C)$.
Since $C$ is an I-cycle, $I[C]$ is a hexagonal system. Then $M_{3}':= M_3|_{I[C]}$ and $M_{1}':= M_1|_{I[C]}$ are two perfect matchings of $I[C]$.
By Lemma \ref{hexagonal system}, $R(I[C])$  has a path
$M_{1}'M_{4}'\cdots M_{t-1}'M_{t}'M_{3}'$.
Let  $M=M_{1}-M_{1}'$. Then
 $P': = (M_{1}=M_{1}'\cup M)(M_{4}'\cup M)\cdots
(M_{t-1}'\cup M)(M_{t}'\cup M)(M_{3}'\cup M)$ is a path between $M_{1}$ and
$M_{3}$ in $R_{6}(G)$.
By Corollary \ref{Necessity}, ${\rm flow}_{L_{i}}(G,M_{3})= {\rm flow}_{L_{i}}(G,M_{1})={\rm flow}_{L_{i}}(G,M_{2})$.  Since $M_{3}\triangle M_{2}=\mathcal{C}\setminus\{C\}$, by the induction hypothesis, $R_{6}(G)$ has a path between $M_{3}$ and $M_{2}$. Thus $R_{6}(G)$ has a path between $M_1$ and
$M_{2}$.

Next suppose that $\mathcal{C}$  contains no I-cycles. We obtain the following critical claim.
\begin{claim}\label{balance}
$M_{1}$ and $M_{2}$ are balanced around each hole of $G$.
\end{claim}

\begin{proof}
Let $h$ be any hole of $G$. \textcolor{black}{By Lemma \ref{tree}, the  rooted tree $T$ in $G^{*}$  contains a directed path $L$ from $h_0^*$ to $h^*$, which is a cut line of $G$, and  $L$  can be divided into several cut segments among  the chosen $n$ cut segments, say $L_i$ in turn, $1\leq i \leq m\leq n$
So $L=L_{1}\cup L_{2}\cup\cdots\cup L_{m-1}\cup L_{m}$, $1\leq m\leq n$, where $h_{0}^{*}$ is a start point of $L_{1}$, $h^*$ is an endpoint  of $L_{m}$ and $L_{i}\cap L_{i+1}=h_{i}^{*}, i=1,\ldots,m-1$ (see Fig. \ref{6-balance}).} \begin{figure}
\centering
\includegraphics{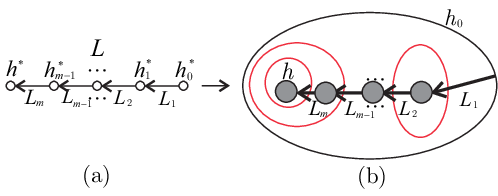}
\caption{\label{6-balance} (a) A cut line  $L$ of  $G$; (b) cut segments of $G$ corresponding to $L$ and some  $(M_{1},M_{2})$-alternating cycles (black).}
\end{figure}
Since ${\rm flow}_{L_i}(G,M_{1})= {\rm flow}_{L_i}(G,M_{2})$ for each $i$, we have
\begin{equation}
   {\rm flow}_{L}(G,M_{1})= \sum_{i=1}^m{\rm flow}_{L_i}(G,M_{1})=\sum_{i=1}^m {\rm flow}_{L_i}(G,M_{2})={\rm flow}_{L}(G,M_{2}).
\end{equation}
By Lemma \ref{EC}, we have ${\rm flow}_{L}(\mathcal{C},M_{1})={\rm flow}_{L}(\mathcal{C},M_{2})$. That is,  
 \begin{equation}\label{3.4}
    {\rm flow}_{L}(\mathcal{C},M_{1})-{\rm flow}_{L}(\mathcal{C},M_{2})=0.
    \end{equation}
Let $C_{1},\ldots,C_{s}$ ($0\leq s\leq k$) be the cycles of $\mathcal{C}$ such that their interiors contain hole $h$ Then the other cycles $C_{s+1},\ldots,C_{k}$ in $\mathcal{C}$  contain no hole $h$ in their interiors. Since both endpoints of $L$ are in the exterior of $C_i$ for $s+1\leq i\leq k$, by Lemma \ref{one II cycle} we have ${\rm flow}_{L}(C_i,M_{1})-{\rm flow}_{L}(C_i,M_{2})=0,$ which together with Eq. (\ref{3.4}) yield

\begin{equation}\label{3.5}
    \sum\limits_{i=1}^{s}({\rm flow}_{L}(C_i,M_{1})-{\rm flow}_{L}(C_i,M_{2}))=0.\end{equation}
Further, by Lemma \ref{one II cycle}, $|{\rm flow}_{L}(C_{j},M_{1})-{\rm flow}_{L}(C_{j},M_{2})|=1$ for $1\leq j\leq s$. So Eq. (\ref{3.5}) implies that
 $s$ is even and a half of cycles of $C_{1},\ldots,C_{s}$ satisfy ${\rm flow}_{L}(C_{j},M_{1})-{\rm flow}_{L}(C_{j},M_{2})=1$, the other half of cycles satisfy ${\rm flow}_{L}(C_{j},M_{1})-{\rm flow}_{L}(C_{j},M_{2})=-1$. By Lemma \ref{two cycle}, a half of cycles of $C_{1},\ldots,C_{s}$ are proper $M_{1}$ (improper $M_{2}$)-alternating cycles and the other  half of cycles are improper $M_{1}$ (proper $M_{2}$)-alternating cycles.
Then $M_{1}$ and $M_{2}$ are balanced around hole $h$ and Claim 1 holds.
\end{proof}

Take all maximal II-cycles  $C_{1},C_{2},\ldots,C_{t}$,  $1\leq t\leq n$, of $\mathcal{C}$ (not contained in the interior of any other cycle of $\mathcal{C}$).  Without loss of generality, suppose $C_{1}$ is a proper $M_{1}$ (improper $M_{2}$)-alternating cycle.

\begin{claim} $I[C_{1}]$ contains an improper $M_{1}$ (proper $M_{2}$)-alternating cycle of $\mathcal{C}$.
\end{claim}
\begin{proof}
Suppose to the contrary that
 $I[C_{1}]$ has no improper $M_{1}$ (proper $M_{2}$)-alternating cycle in $\mathcal C$.   Since $C_{1}$ is a   II-cycle, $I[C_{1}]$ contains at least  a hole $h$. Then there is no improper $M_{1}$ (proper $M_{2}$)-alternating cycle in $\mathcal C$ with hole $h$ in its interior.
As $C_{1}$ is a proper $M_{1}$ (improper $M_{2}$)-alternating cycle, $M_{1}$ and $M_{2}$ are not balanced around $h$.
This is a contradiction to Claim \ref{balance}.
\end{proof}
\begin{figure}
\centering
\includegraphics[scale=0.8]{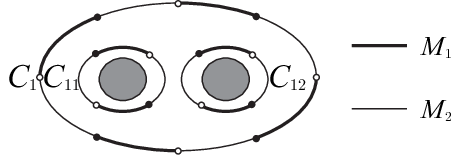}
\caption{\label{Figure_7}  A proper (resp. improper) $M_{1}$-alternating cycles $C_{1}$ (resp. $C_{11}$ and $C_{12}$).}
\end{figure}
By Claim 2, take all maximal improper $M_{1}$ (proper $M_{2}$)-alternating cycles
$C_{11},\ldots,C_{1r}$ ($r\geq 1$) of $\mathcal{C}$ in $I[C_{1}]$.  Then $C_1$ and the $C_{1j}$'s are mutually disjoint, and  the the \textcolor{black}{multiple} connected region \textcolor{black}{between} $C_1$ and such the $C_{1j}$'s contains no  holes of $G$ (see Fig. \ref{Figure_7}). Otherwise, a contradiction to Claim 1 happens.
Let $G_{1}$ be a subgraph of $G$ obtained from $I[C_{1}]$ by deleting the interiors  of $C_{11}, C_{12}, \ldots, C_{1r}$. 
\textcolor{black}{By Lemma \ref{equivalent}}, $G_{1}$ is a coronoid system with boundary  $\mathcal{C}_{1}':= C_{1}\cup C_{11}\cup\cdots \cup C_{1r}$.
Let $M_{1}'=M_{1}\triangle E(\mathcal{C}_{1}')$ and $M_{1}^{*}=M_{1}|_{G_{1}}$, $M_{2}^{*}=M_{1}'|_{G_{1}}$. Then $M_{1}^{*}\triangle M_{2}^{*}=\mathcal{C}_{1}'$. Lemmas \ref{nice cycle} and \ref{elementary} imply that $G_{1}$ is an elementary coronoid system. \textcolor{black}{By Lemma \ref{path}, $R_6(G_{1})$ has a path between $M_{1}^{*}$ and  $M_{2}^{*}$.} 
By the same manner as the beginning of this proof,   $R_{6}(G)$ has a path between  $M_{1}$ and  $M_{1}'$. By Corollary \ref{Necessity},
${\rm flow}_{L_{i}}(G,M_{1}')= {\rm flow}_{L_{i}}(G,M_{1})={\rm flow}_{L_{i}}(G,M_{2})$.
Since $M_{1}'\triangle M_{2} =\mathcal{C}\setminus\mathcal{C}_{1}'$, by induction hypothesis, $R_{6}(G)$ have a path between $M_{1}'$ and $M_{2}$.
Thus, $R_{6}(G)$ has a path between $M_1$ and $M_{2}$.
\end{proof}

\begin{color}{black}
\begin{rem}
{\rm If we can find $n$ pairwise internally disjoint cut segments in rooted tree $T$ (the  branch vertices  correspond only to non-hexagonal faces), then Theorem \ref{main results} remains valid when their orientations may be arbitrarily  assigned  (for example, see Fig. \ref{Figure_13}). To this end  we only need to make some minor changes in    the proof (Claim 1) of Lemma \ref{Sufficiency}: a  path $L$ in $T$ between  $h_0^*$ to $h^*$ consists of  cut segments among  the chosen $n$ cut segments,  reversing  the directions of some cut segments to get a directed path (cut line).  The following  Lemma \ref{anti-oriented} shows that the flows of two perfect matchings remains the same  across a pair of opposite  cut segments.   However, it is not clear at present whether such disjoint cut segments always exist.}
\end{rem}
\end{color}

\begin{figure}
\centering
\includegraphics[scale=0.8]{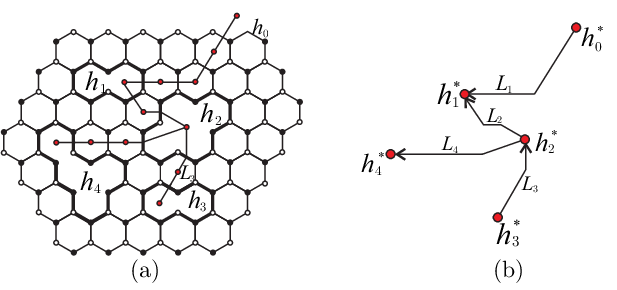}
\caption{\label{Figure_13} (a) A coronoid system $G$ and a tree of dual graph $G^{*}$, (b) four chosen internally disjoint cut segments $L_{1}, L_{2},L_{3}$ and $L_{4}$. }
\end{figure}

\begin{lem}\label{anti-oriented}
Let $L$ be a cut segment of $G$ and \textcolor{black}{$L'$ a cut segment obtained from $L$  in opposite direction.}
If ${\rm flow}_{L}(G,M_{1})= {\rm flow}_{L}(G,M_{2})$,
then ${\rm flow}_{L'}(G,M_{1})= {\rm flow}_{L'}(G,M_{2})$.
\end{lem}
\begin{proof}
Suppose ${\rm flow}_{L}(G,M_{1})= {\rm flow}_{L}(G,M_{2})$.
Since $L'$ is in the opposite direction to $L$, $(L'^{*})^{+}=(L^{*})^{-}$ and $(L'^{*})^{-}=(L^{*})^{+}$. By Eq. (\ref{dif}), we have ${\rm flow}_{L'}(G,M_{i})=|(L^{*})^{-}\cap M_{i}|-|(L^{*})^{+}\cap M_{i}|=-{\rm flow}_{L}(G,M_{i})$, $i=1,2$.
Therefore, ${\rm flow}_{L'}(G,M_{1})= {\rm flow}_{L'}(G,M_{2})$.
\end{proof}

\section{Applications}

As  applications of Theorem \ref{main results} we first give a criterion for \textcolor{black}{the connectedness of the resonance graph of a coronoid system. Further, we discuss the resonance graph of a nanotube.}

\subsection{\textcolor{black}{Coronoid systems}}

By Theorem \ref{main results}, we can obtain a simple criterion to determine whether
the resonance graph of a coronoid system is connected  as follows.

\begin{thm}{\label{connected}}
Let $G$ be a matchable coronoid system. Then $R_{6}(G)$ is connected if and only if each nice cycle of $G$ is an I-cycle.
\end{thm}
\begin{proof}
Necessity: Assume $R_{6}(G)$ is connected. Suppose, to the contrary, that
there exists a nice cycle $C$ that is a II-cycle. Then $G$ has a perfect matching $M_{1}$ such that $C$ is $M_{1}$-alternating. Let $M_{2}:=M_{1}\triangle E(C)$. Then $M_{2}$ \textcolor{black}{is also} a perfect matching of $G$, and
there exists a cut segment $L$ of the  $n$ chosen cut segments joining an end point in the interior of $C$ and other an end point in the exterior of $C$. By Lemma \ref{one II cycle},
${\rm flow}_{L}(C,M_{1})\neq {\rm flow}_{L}(C,M_{2})$. By Lemma \ref{EC}, ${\rm flow}_{L}(G,M_{1})\neq {\rm flow}_{L}(G,M_{2})$.
By Theorem \ref{main results}, $M_{1}$ and
$M_{2}$ are not in the same connected component, a contradiction.

Sufficiency: Assume $M$, $M'$ are any two perfect matchings of $G$ and  $\mathcal{C}=M\triangle M'=C_{1}\cup\cdots\cup C_{k}$. Such $C_{i}$'s are I-cycles as each nice cycle of $G$ is an I-cycle.
So both end points of any cut segment $L$ are in the exterior of $C_{i}$. By Lemma \ref{one II cycle},
${\rm flow}_{L}(C_{i}, M)={\rm flow}_{L}(C_{i}, M')$. 
By Lemma \ref{EC}, ${\rm flow}_{L}(G, M)={\rm flow}_{L}(G,M')$.
Theorem \ref{main results} shows that  $M$ and $M'$ lie in connected component of $R_{6}(G)$. Hence $R_{6}(G)$ is connected.
\end{proof}

For an elementary coronoid system $G$, the boundary of each hole of $G$ is a nice II-cycle by Lemma \ref{elementary}. So by Theorem \ref{connected} we have an immediate consequence  as follows.
\begin{cor}\label{coronoid}
Let $G$ be \textcolor{black}{an} elementary coronoid system. Then $R_{6}(G)$ is disconnected.
\end{cor}

\begin{figure}
\centering
\includegraphics[scale=0.8]{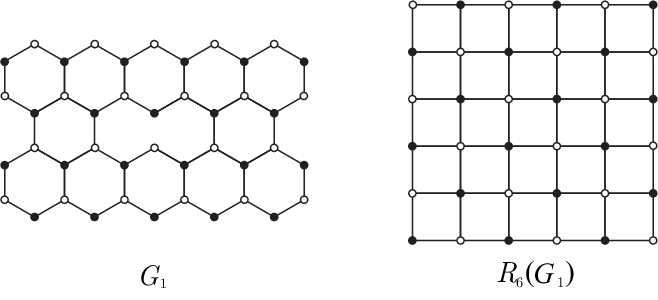}
\caption{\label{Figure_8}  A coronoid system with its resonance graph, a $5\times 5$ chessboard.}
\end{figure}

\begin{ex}
{\rm Let $G_1$ be \textcolor{black}{a} single coronoid system in  Fig. \ref{Figure_8}. Then four vertical edges of $G_1$ in the central level are forbidden edges. So each nice cycle of $G_1$ is an I-cycle and $R_6(G_1)$ is connected. In fact  $R_{6}(G_1)=P_6\Box P_6$, a $5\times 5$ chessboard.}
\end{ex}

\begin{figure}
\centering
\includegraphics[scale=0.8]{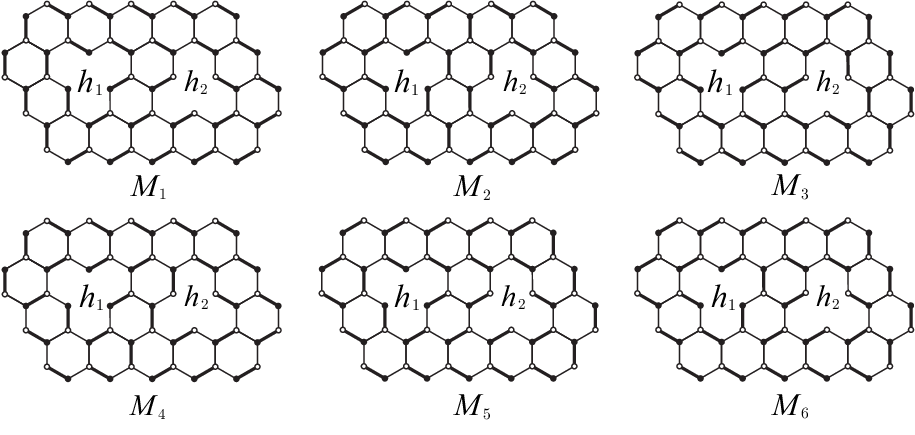}
\caption{\label{Figure_9} Six perfect matchings which  come from the  different components of  $R_6(G_{2})$.}
\end{figure}

\begin{ex}
{\rm Let $G_2$ be a coronoid system with $2$ holes in Fig. \ref{Figure_9}.  Since the boundaries of infinite face and $h_{1}$ are $M_{1}$-alternating cycles and the boundary of $h_{2}$ is $M_{4}$-alternating cycle, Lemmas \ref{nice cycle} and \ref{elementary} imply that $G_{2}$ is elementary. Then
 its resonance graph $R_{6}(G_2)$ is disconnected. With the computer, we obtain that $R_{6}(G_2)$ contains $334$ vertices and consists of  $3$ singletons $M_{1}$, $M_{2}$ and $M_{3}$ and $3$ non-singleton components that  contain $M_{4}$, $M_{5}$, $M_{6}$, respectively.}
\end{ex}

\subsection{Nanotubes}

We consider open-ended single-walled \emph{nanotubes} (nanotubes for short). A nanotube is a part of hexagonal tessellation of a cylinder, which can be formed by rolling up a two-dimensional hexagonal sheet.  More precisely, we define a nanotube as follows.
\begin{figure}[!htbp]
\centering
\includegraphics[scale=0.7]{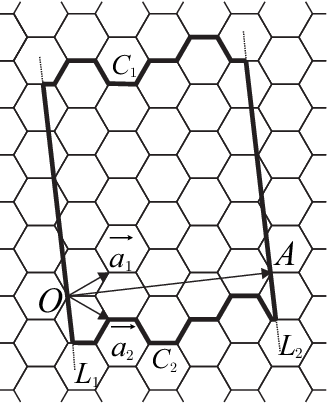}
\caption{\label{Figure_10}  Illustration of a $(3,2)$-type nanotube.}
\end{figure}
Choose any lattice point in the plane hexagonal lattice as the origin $O$. Let $\vec{a_{1}}$ and $\vec{a_{2}}$ be the two basic lattice vectors (see Fig. \ref{Figure_10}). Choose a vector
$\overrightarrow{OA}=n \vec{a_{1}}+m\vec{a_{2}}$ such that $n$ and $m$ are two integers and at least one of them is not zero. Draw two straight lines $L_1$ and $L_2$ passing through $O$ and $A$ perpendicular to $\overrightarrow{OA}$, respectively. By rolling up the hexagonal strip between $L_1$ and $L_2$ and gluing $L_1$ and $L_2$ such that $A$ and $O$ superimpose, we can obtain a hexagonal tessellation $H$ of the cylinder. $L_1$ or $L_2$ indicate the direction of the axis of the cylinder. Using the terminology of graph theory, a nanotube $N$ is defined to be the finite graph formed  by all the hexagons of $H$  between  two disjoint cycles $C_1$ and $C_2$ (length  at least is 4) of $H$ encircling the axis, called  an $(n,m)$-type nanotube. The cycles $C_1$ and $C_2$ are two open-ends of $N$.

Tratnik et al. investigated the resonance graph of a nanotube as a natural extension of the resonance graph of a coronoid  system, where only the two open-ends (corresponding to hole and the exterior face) are not allowed to flip even if an end of a nanotube is a cycle of length 6.  They proposed the following conjecture.

\begin{conj}(\cite{TZra15}){\label{conj}}
The resonance graph  $R_6(N)$ of a matchable nanotube $N$ is not connected.
\end{conj}


\begin{figure}
\centering
\includegraphics[scale=0.7]{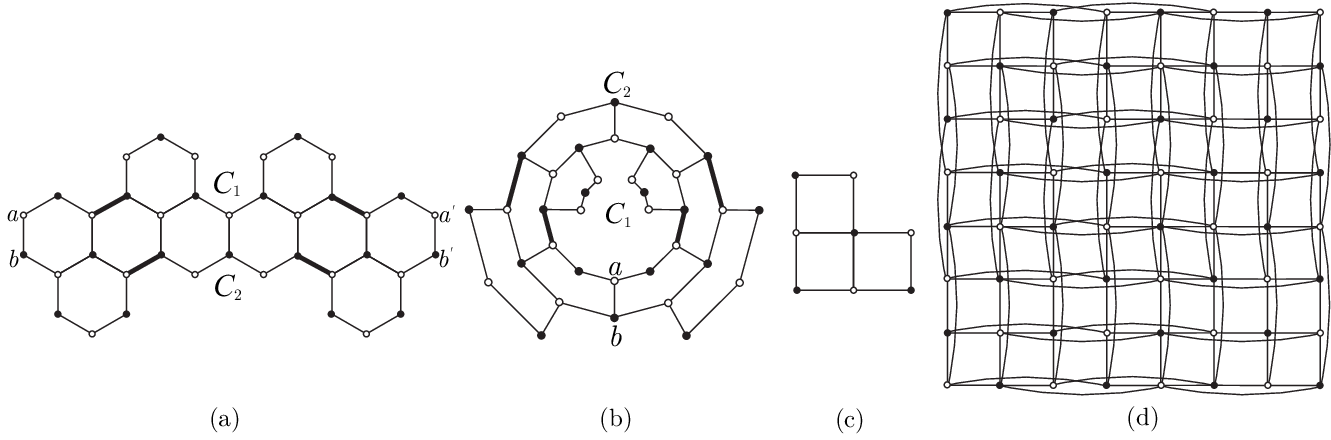}
\caption{(a) A nanotube $N_1$, (b) \textcolor{black}{a plane drawing of $N_{1}$}, (c) the resonance graph $R_6(H_1)$, (d) the resonance graph $R_6(N_1)$.\label{Figure_11}}
\end{figure}

However, we give an example of a nanotube whose resonance graph is connected, which shows that Conjecture \ref{conj} is not true.
\begin{ex}
{\rm Let $N_{1}$ be a nanotube obtained from the hexagonal system shown in Fig. \ref{Figure_11} by gluing edges $ab$ and  $a'b'$. Then $N_{1}$ has four forbidden edges (thick edges).  We can obtain two isomorphic hexagonal chains of 4 hexagons, say $H_{1}$ and $H_{2}$, from $N_{1}$ by the removal of the four forbidden edges. The $H_{1}$ and $H_{2}$ have isomorphic resonance graphs. Since  the resonance graphs $R_{6}(H_{1})$ and $R_{6}(H_{2})$ are connected, by Lemma \ref{component}, $R_{6}(N_{1})\cong R_{6}(H_{1})\square R_{6}(H_{2})$  is connected. In this way we can construct infinite many such examples.}
\end{ex}

\textcolor{black}{Note that a nanotube is also a planar bipartite graph \cite{SHZ96,TZra15}. Next we always draw a nanotube $N$ in the plane  so that one end corresponds to a hole and the other end to the exterior face and all the other faces are hexagons. For examples,  see Figs. \ref{Figure_11}(b) and \ref{Figure_12}(b).}  However,  a nanotube drawn in the plane cannot be a coronoid system.

Theorem 2.3 in  \cite{ZZYh04} implies the following result, which shows that Conjecture \ref{conj} is true for elementary nanotubes.
\begin{lem}\label{elem-N} Let $N$ be an elementary nanotube. Then $R_6(N)$ is disconnected.
\end{lem}

A cycle of a nanotube $N$ is also an I-cycle or II-cycle, depending on whether its interior contains the hole.
Next we will give a simple criterion for a nanotube  to have  connected resonance graph by using resonance graph of a plane bipartite graph (see Lemma \ref{component}).
\begin{thm}{\label{N-connected}}
Let $N$ be a matchable nanotube. Then $R_{6}(N)$ is connected if and only if each nice cycle of $N$ is an I-cycle or bounds a hexagonal system.
\end{thm}
\begin{proof} Let $G_1,G_2,\ldots, G_k$ denote the elementary components of $N$, $k\geq 1$.  Since each $G_i$ is a plane elementary bipartite graph, $G_i$ is a $K_2$ (a complete graph with two vertices) or is 2-connected. For the former, $R_6(G_i)$ is a singleton. For the latter, $G_i$ is either a hexagonal system or a subnanotube on $N$ (its ends may intersect). If $G_i$ is not a hexagonal system, then it has an interior face containing the hole of $N$ and the other interior faces are hexagonal faces on $N$; Otherwise, let $C$ be the boundary of a non-hexagonal interior face  $f$ that does not contain the hole of $N$. Then $I[C]$ in $N$ is a hexagonal system. By Lemma \ref{nice cycle} $C$ is a nice cycle of $G_i$, $I[C]$ and  $N$. By Lemma \ref{elementary}  $I[C]$ is an elementary hexagonal system, which contains no forbidden edges, a contradiction. Hence by Lemma \ref{component}, we have that  $R_6(N)\cong R_6(G_1)\square \cdots \square R_6(G_k)$. If $G_i$ is a hexagonal system, then $R_6(G_{i})$ is connected. If $G_i$ is an elementary nanotube, then $R_6(G_i)$ is disconnected by Lemma \ref{elem-N}.  In this case $G_i$ contains a nice II-cycle of $N$. So $R_6(N)$ is connected if and only if each $G_i$ is $K_2$ or a hexagonal system, equivalently, $N$ has no nice II-cycle.
\end{proof}
\begin{rem}
{\rm  Each 2-connected elementary component of a matchable nanotube is either a hexagonal  system or a degenerated nanotube (two ends may have a common part).}
\end{rem}
So we now wonder whether two perfect matchings of a nanotube $N$ lie in the same component of its resonance graph.

\begin{color}{black}
For a nanotube $N$,  a \emph{cut segment} $L$ is a  path in dual graph $N^*$ with a direction between the vertices corresponding two open ends. The flow of a perfect matching $M$ across $L$ are defined similarly as coronoid system. For example, see \ref{Figure_12}(b) and (d).


\end{color}

We find that the necessity of Theorem \ref{main results} holds for nanotubes. However the sufficiency no longer holds; \textcolor{black}{see} the following example.
\begin{figure}
\centering
\includegraphics[scale=0.7]{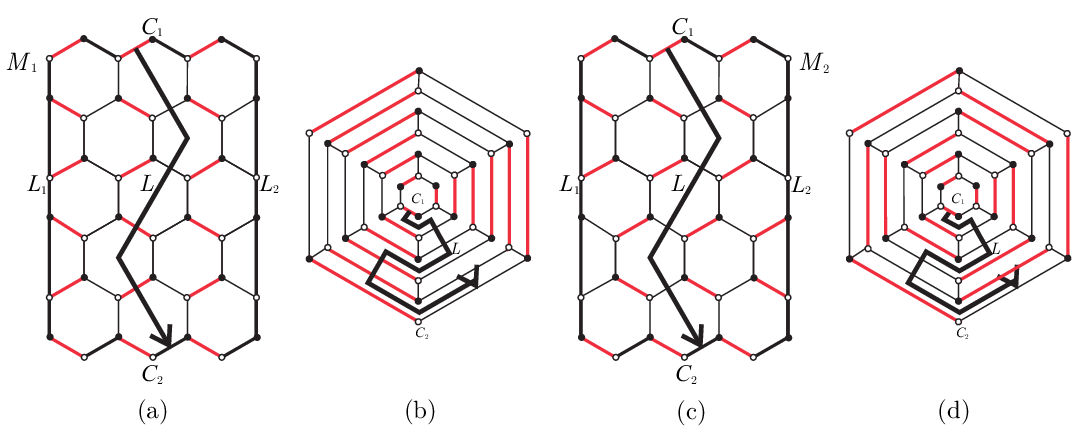}
\caption{\label{Figure_12} An example of non-elementary nanotube.}
\end{figure}
\begin{ex}
{\rm Let $N$ be a (3,0)-type nanotube as shown in Fig. \ref{Figure_12}.  Let $M_{1}$ and $M_{2}$ be two distinct perfect matchings of $N$ and $L$ be a cut segment. Then all vertical edges of $N$ are forbidden edges. So $N$ is not elementary. It is easy to see that ${\rm flow}_{L}(N, M_{1})=-1={\rm flow}_{L}(N,M_{2})$, but $M_{1}$ and $M_{2}$ are isolated vertices of $R_{6}(N)$.}
\end{ex}
For an elementary nanotube, the following result shows that Theorem \ref{main results} also holds.

\begin{thm}
   {\label{N-flow}}
Let $N$ be an elementary nanotube. 
Two perfect matchings lie in the same connected component of the resonance graph of $N$ if and only if their flows across any cut segment $L$ are equal.
\end{thm}
\begin{proof}   \textcolor{black}{It is easy to check that Lemmas \ref{one II cycle} to \ref{EC} and Corollary \ref{Necessity}} also hold for a cut segment $L$ of a matchable nanotube $N$.
Thus, the necessity holds. We can outline a proof to the sufficiency by an analogous way to Lemma \ref{Sufficiency}.   For  perfect matchings  $M_{1}$ and $M_{2}$ of $N$, suppose ${\rm flow}_{L}(N,M_{1})= {\rm flow}_{L}(N,M_{2})$ for  cut segment $L$. Let $\mathcal C=M_1\triangle M_2$, which consists of disjoint $(M_1,M_2)$-alternating cycles. We now show that $R_{6}(N)$ has a path between $M_{1}$ and $M_{2}$ by induction on the number of cycles in $\mathcal{C}$. Without loss of generality, suppose $\mathcal C$ has no I-cycles. Lemma \ref{EC} implies that ${\rm flow}_{L}(\mathcal{C}, M_{1})={\rm flow}_{L}(\mathcal{C},M_{2})$.  \textcolor{black}{By Lemma \ref{two cycle}},
we have a simpler way to show that  $M_{1}$ and $M_{2}$ are balanced around the hole of $N$.  Take a proper $M_1$-alternating cycle $C_1$ and an improper $M_1$-alternating cycle $C_2$ in $\mathcal C$. We use $N[C_1, C_2]$ to denote subnanotube with two open ends $C_1$ and $C_2$ in $N$.
Let $M_1^*=M_1\triangle E(C_1)\triangle E(C_2)$, $M_{1}':= M_1|_{N[C_{1},C_{2}]}$ and $M_{1}'':= M_1^*|_{N[C_{1},C_{2}]}$. Then $M_1'$ and $M_1''$ are two perfect matchings of $N[C_{1},C_{2}]$ and $M_1'\triangle M_1''=C_{1}\cup C_{2}$.  \textcolor{black}{Note that $N[C_{1},C_{2}]$ is also plane bipartite graph. Without loss of generality, let $I[C_{2}]\subset I[C_{1}]$. Clearly, $C_{1}$ is a nice cycle of $N$. Since $N$ is plane elementary bipartite, by Lemma \ref{I[C] elementary}, $I[C_{1}]$ of $N$ is a plane elementary bipartite graph. Note that $C_{2}$ is a nice cycle of $I[C_{1}]$. By Lemma \ref{I[C] elementary} again,  $O[C_{2}]$ of $I[C_{1}]$ is elementary, that is, the graph  $N[C_{1},C_{2}]$ obtained from $I[C_{1}]$ deleting the interior of $C_{2}$ is plane elementary bipartite graph } (in a non-elementary case, this result does not hold; for example, see Fig. \ref{Figure_12}).
\textcolor{black}{By Lemma \ref{path}, $R_6(N[C_{1},C_{2}])$ has a path between $M_{1}'$ and  $M_{1}''$.}  
By the same manner as in the proof of Lemma \ref{Sufficiency},  $R_{6}(N)$ has a path between  $M_{1}$ and  $M_{1}^{*}$. \textcolor{black}{Since}  $M_1^*\triangle M_2=\mathcal C\setminus \{C_1,C_2\}$,
by the induction hypothesis,  $R_6(N)$ has a path from $M_1^*$ to $M_2$. Hence $R_6(N)$ has a path from $M_1$ to $M_2$.
\end{proof}


\vspace{5pt}
\noindent{\textcolor{black}{\textbf{Acknowledgments}}}

\textcolor{black}{The authors would like to thank the two anonymous referees for their careful reading to the manuscript and their valuable comments and suggestions in improving this manuscript.}

\end{document}